\documentclass[12pt]{article}

\usepackage{fullpage}
\usepackage{graphicx}
\usepackage{amsthm}
\usepackage{epsfig}
\usepackage{amsfonts}
\usepackage{amsmath}
\usepackage{enumitem}
\usepackage{mathrsfs}

\usepackage{geometry}
\usepackage{tikz}
\usepackage{comment}

\title{Packing triangles in weighted graphs}

\author{
   Guillaume Chapuy\thanks{Current address: CNRS UMR 7089, Universit\'{e} Paris Diderot -- Paris 7, Case 7014, 75205 Paris Cedex 13, France.  Previously supported by a PIMS-CNRS postdoctoral fellowship.}
\and
   Matt DeVos\thanks{Supported in part by an NSERC Discovery Grant (Canada) and a Sloan Fellowship.}
\and
   Jessica McDonald\thanks{Current address: Department of Mathematics and Statistics, Auburn University, Auburn, AL, USA 36849.}
\and
   Bojan Mohar\thanks{Supported in part by an NSERC Discovery Grant (Canada), by the Canada Research Chair program, and by the Research Grant P1--0297 of ARRS (Slovenia).}~\thanks{On leave from: IMFM \& FMF, Department of Mathematics, University of Ljubljana, Ljubljana, Slovenia.}
\and
   Diego Scheide\thanks{Postdoctoral fellowship at Simon Fraser University, Burnaby.}
}

\date{}

\def\cA{\mathcal{A}}
\def\cB{\mathcal{B}}

\def\cI{\mathcal{I}}
\def\cJ{\mathcal{J}}
\def\cK{\mathcal{K}}
\def\cT{\mathcal{T}}
\def\sm{\setminus}
\newcommand{\set}[2]{\{#1 : #2\}}

\newcommand{\drawvertex}[1]{\filldraw[black] #1 circle (2mm); \draw[thick] #1 circle (2mm)}
\newcommand{\drawtext}[2]{\draw #1 node {\small #2}}

\theoremstyle{plain}
\newtheorem{theorem}{Theorem}[section]
\newtheorem{lemma}[theorem]{Lemma}
\newtheorem{corollary}[theorem]{Corollary}

\newtheorem{conjecture}[theorem]{Conjecture}

\theoremstyle{definition}
\theoremstyle{remark}

\newcommand\DEF[1]{\emph{#1\/}}
\newcommand\one{\mathbf{1}}

\begin{document}

\maketitle
\vspace{-1cm}
\begin{center}
  {Department of Mathematics}\\
  {Simon Fraser University}\\
  {Burnaby, B.C. V5A 1S6}
\end{center}

\begin{abstract}
  Tuza conjectured that for every graph $G$, the maximum size $\nu$ of a set of edge-disjoint triangles and minimum size $\tau$ of a set of edges meeting all triangles satisfy $\tau \leq 2\nu$. We consider an edge-weighted version of this conjecture, which amounts to packing and covering triangles in multigraphs. Several known results about the original problem are shown to be true in this context, and some are improved. In particular, we answer a question of Krivelevich who proved that $\tau \leq 2\nu^*$ (where $\nu^*$ is the fractional version of $\nu$), and asked if this is tight. We prove that $\tau \leq 2\nu^*-\frac{1}{\sqrt{6}}\sqrt{\nu^*}$ and show that this bound is essentially best possible.
\end{abstract}

\section{Introduction}

We shall assume in this paper that graphs are simple, and use the term \DEF{multigraph} when parallel edges are permitted.  Let $G = (V,E)$ be a graph and let $\cT=\cT(G)$ be the set of triangles of $G$. A \DEF{packing} is a set of edge-disjoint triangles and a \DEF{transversal} is a set of edges which meets every triangle. We define the following parameters:
\begin{align*}
\nu(G) &= \max\{|Z| : Z\subseteq \cT(G) \mbox{ is a packing in $G$}\}\quad \mbox{and} \\
\tau(G) &= \min\{|F| : F\subseteq E(G) \mbox{ is a transversal in $G$}\}.
\end{align*}
These are the usual packing and transversal parameters for the hypergraph with vertex set $E$ and hyperedges corresponding to $\cT$.

It is immediate that $\nu(G) \le \tau(G) \le 3 \nu(G)$ since, given a maximum set $\cT'$ of $\nu(G)$ edge-disjoint triangles, every transversal must contain at least one edge from each triangle in $\cT'$, and on the other hand, the set of all edges in $\cT'$ is a transversal. The following conjecture was proposed by Tuza \cite{Tuza1} in 1981. It asserts that the trivial upper bound $3\nu$ on $\tau$ can be improved.

\begin{conjecture}[Tuza]
\label{conj:tuza}
~$2 \nu(G) \ge \tau(G)$ for every graph $G$.
\end{conjecture}

It is worthwhile to interpret $\nu(G)$ and $\tau(G)$ as solutions to integer programs, so let us pause to do so now.    Let $A$ be the edge-triangle incidence matrix of $G$, i.e., $A_{e,t} = 1$ if the triangle $t$ contains the edge $e$ and otherwise $A_{e,t} = 0$. Then we have:
\begin{align*}
\nu(G) &= \max \{ \langle\one,x\rangle : \mbox{$Ax \le \one$ and $x \in {\mathbb Z}^\cT_+$} \}, \\
\tau(G) &= \min\{ \langle\one,y\rangle : \mbox{$A^{\top}y \ge \one$ and $y \in {\mathbb Z}_+^E$} \},
\end{align*}
where $\one\in{\mathbb Z}^\cT$ or $\one\in{\mathbb Z}^E$ denotes the all-1 function (it is clear from the context which of the two possibilities applies), and the inner product $\langle\cdot,\cdot\rangle$ is the usual one, $\langle u,v\rangle=\sum_{t\in\cT}u(t)v(t)$ if $u,v\in{\mathbb Z}_+^\cT$ or $\sum_{e\in E}u(e)v(e)$ if $u,v\in{\mathbb Z}_+^E$.

Relaxing the integrality constraints, we find the following dual linear programs:
\begin{align}
\nu^*(G) &= \max \{ \langle\one,x\rangle : \mbox{$Ax \le \one$ and $x \in {\mathbb R}^T_+$} \},
\label{eq:s1}\\
\tau^*(G) &= \min\{ \langle\one,y\rangle : \mbox{$A^{\top}y \ge \one$ and $y \in {\mathbb R}_+^E$} \},
\label{eq:s2}
\end{align}
whose optimal values $\nu^*(G)$ and $\tau^*(G)$ are called the \DEF{fractional packing number} and \DEF{fractional transversal number}, respectively. This gives us the following meaningful chain of inequalities:
\[ \tau(G) \ge \tau^*(G) = \nu^*(G) \ge \nu(G). \]
Although Tuza's Conjecture~\ref{conj:tuza} remains wide open, there have been a number of useful partial results.  Below we highlight three of these.

\begin{theorem}[Krivelevich \cite{Kri}]
\label{thm:Kriv}
For every graph $G$ we have:
\begin{enumerate}[label=\textup{(\roman*)}, widest=iii, leftmargin=*]
\item $2 \nu(G) \ge \tau^*(G)$.
\item $2 \nu^*(G) \ge \tau(G)$.
\end{enumerate}
\end{theorem}

\begin{theorem}[Tuza \cite{Tuza2}]
\label{thm:tuza planar}
Conjecture~\ref{conj:tuza} holds whenever $G$ is planar.
\end{theorem}

\begin{theorem}[Haxell \cite{Hax}]
\label{thm:haxell}
For every graph $G$, we have $~2.87\, \nu(G) \ge \tau(G)$.
\end{theorem}

There is a natural weighted analogue of Tuza's triangle packing problem. Namely, if $w \colon E \to {\mathbb Z}_+$ is an \DEF{edge-weighting}, then  using the edge-triangle matrix $A$ introduced above, we define
\begin{align*}
\nu_w(G) &= \max \{ \langle\one,x\rangle : \mbox{$Ax \le w$ and $x \in {\mathbb Z}^T_+$} \}, \\
\tau_w(G) &= \min\{ \langle w,y\rangle : \mbox{$A^{\top}y \ge \one$ and $y \in {\mathbb Z}_+^E$} \}.
\end{align*}
In other words, $\nu_w(G)$ is the largest number of (not necessarily distinct) triangles such that each edge $e$ is contained in at most $w(e)$
of them, and we say that such a collection of triangles is a (weighted) \DEF{packing}. Similarly, $\tau_w(G)$ is the minimum weight of a transversal, where the \DEF{weight} of an edge-set $R$ is defined as the sum of the weights of its elements, $w(R) = \sum_{e\in R} w(e)$.

As before, relaxing the integrality constraints gives us dual linear programs:
\begin{align}
\nu_w^*(G) &= \max \{ \langle\one,x\rangle : \mbox{$Ax \le w$ and $x \in {\mathbb R}^T_+$} \}
\label{eq:s3}\\
\tau_w^*(G) &= \min\{ \langle w,y\rangle : \mbox{$A^{\top}y \ge \one$ and $y \in {\mathbb R}_+^E$} \}
\label{eq:s4}
\end{align}
and we have the chain of inequalities:
\[ \tau_w(G) \ge \tau_w^*(G) = \nu_w^*(G) \ge \nu_w(G). \]

Admissible solutions $x$ and $y$ to the linear programs \eqref{eq:s3} and \eqref{eq:s4} are called \DEF{fractional packings} and \DEF{fractional transversals}, respectively.

Given a weighting $w$ of a graph $G$, we can define a multigraph $G'$ by replacing each edge $e$ in $G$ with $w(e)$ parallel edges. We consider a triangle in $G'$ to be a $K_3$-subgraph of $G'$ (i.e. no multiple edges), and define the packing and transversal numbers $\nu$ and $\tau$ for $G'$ accordingly. Any weighted packing in $G$ corresponds naturally to a packing of same size in $G'$ and vice versa, implying that $\nu_w(G)=\nu(G')$. Also any transversal with weight $k$ in $G$ corresponds naturally to a transversal of size $k$ in $G'$, but the other direction is not generally true. However, if $C$ is any optimal transversal in $G'$ and $e\in C$ then $C$ also contains all edges that are parallel to $e$. Hence $C$ corresponds naturally to a transversal of weight $|C|$ in $G$. Consequently, we have $\tau_w(G)=\tau(G')$. Similarly, the fractional packing and covering parameters for $(G,w)$ and $G'$ are the same. Thus it is admissible to investigate packings and transversals in multigraphs instead of the weighted problems in simple graphs. We will do so in Sections \ref{sec:nuvstaustar} and \ref{sec:tau vs nu}.

The subject of this paper is the following weighted version of Tuza's conjecture.

\begin{conjecture}
For every graph $G= (V,E)$ and $w\colon E \to {\mathbb Z}_+$, we have $$2 \nu_w(G) \ge \tau_w(G).$$
\end{conjecture}

First we generalize Krivelevich's Theorem \ref{thm:Kriv} to the weighted case. For Krivelevich's result, the inequality between $\nu$ and $\tau^*$ is tight for $K_4$ and we show that the same bound holds in the weighted case.  On the other hand, the inequality between $\tau$ and $\nu^*$ is not tight and we show that an improvement can be made.

\begin{theorem}\label{thm:wKriv}
For every graph $G=(V,E)$ and $w\colon E \rightarrow {\mathbb Z}_+$ we have
\begin{enumerate}[label=\textup{(\roman*)}, widest=iii, leftmargin=*]
\item\label{subthm:wKriv:1}
  $\tau_w(G) \le 2 \tau_w^*(G) - \sqrt{\tau_w^*(G)/6}+1$, and
\item\label{subthm:wKriv:2}
  $2 \nu_w(G) \ge \tau_w^*(G)$.
\end{enumerate}
\end{theorem}

Although \ref{subthm:wKriv:1} may appear to be a rather small improvement on Krivelevich's original result, we show that this improvement is best possible up to a logarithmic factor (even for the unweighted case). See Section \ref{sec:nuvstaustar}. This answers a question of Krivelevich about the tightness of his bounds.

We also prove weighted analogues of Tuza's and Haxell's theorems in Sections \ref{sec:surfaces} and \ref{sec:tau vs nu}, respectively.
We shall extend Tuza's Theorem \ref{thm:tuza planar} to weighted graphs embedded in an arbitrary surface. We refer to \cite{MT} for standard terminology concerning graphs on surfaces. A cycle $C$ of a graph embedded in a surface is said to be \DEF{surface-separating} if cutting the surface along $C$ disconnects the surface. Note that every \DEF{facial cycle} (i.e. a cycle bounding a face) is surface-separating.

\begin{theorem}
\label{thm:surface}
Suppose that a graph $G$ is embedded in a surface such that every triangle is surface-separating.
Then for every $w\colon E \to {\mathbb Z}_+$, we have $$2\, \nu_w(G) \ge \tau_w(G).$$
\end{theorem}

Since every cycle in a graph embedded in the plane is surface-separating, Theorem \ref{thm:surface} yields the following extension of Tuza's Theorem \ref{thm:tuza planar}.

\begin{corollary}
\label{cor:wtuza}
For every planar graph $G = (V,E)$ and every $w\colon E \to {\mathbb Z}_+$, we have
$~2\, \nu_w(G) \ge \tau_w(G)$.
\end{corollary}

Interestingly, the introduction of weights seems to simplify the proof of Tuza's theorem, while it appears to make things more difficult for Haxell's theorem (in fact we get a slightly larger constant factor).

\begin{theorem}\label{thm:whaxell}
For every graph $G$ and $w\colon E \rightarrow {\mathbb Z}_+$, we have $$2.92\, \nu_w(G) \ge \tau_w(G).$$
\end{theorem}

When considering fractional versions of integer programming problems, it is natural to wonder ``how fractional'' optimal solutions need to be. For instance, it is immediate that three perfect matchings pack in a bridgeless cubic graph $H$ if and only if $H$ is $3$-edge-colourable, and Edmonds' matching polytope theorem \cite{Ed} proves that there is always a fractional packing of value $3$. The Berge-Fulkerson conjecture asserts that there always exists a half-integral packing of value $3$.  Numerous other theorems and conjectures in combinatorial optimization concern this phenomena (see also \cite{Hu}, \cite{LS}, \cite{Se}). Returning to our problem of packing and covering triangles, for any positive integer $k$, let $\nu_k = \nu_w$ ($\tau_k = \tau_w$) where $w\colon E \to {\mathbb Z}_+$ is the constant function of value $k$. It is immediate from the rationality of the matrix $A$ that for every graph $G$ there exists an integer $k$ so that $\nu^*(G) = \frac{1}{k}\,\nu_k(G)$ and $\tau^*(G) = \frac{1}{k}\, \tau_k(G)$.  The question which arises is whether or not there exists a fixed integer $k$ so that
$\frac{1}{k}\nu_k(G) = \nu^*(G)$ for every graph $G$ (i.e. whether or not there is a fixed $k$ so that there always exists an optimal fractional packing assigning rationals with denominator $k$).  We resolve this question in the negative with the following theorem, proved in Section \ref{sec:integrality}.

\begin{theorem}
\label{thm:integrality}
There does not exist a fixed integer $k$ so that  $\nu^*(G) = \frac{1}{k}\nu_k(G)$ for every graph $G$, and similarly there is no fixed $k$ so that $\tau^*(G) = \frac{1}{k}\tau_k(G)$ for every graph $G$.
\end{theorem}

\section{Comparing $\tau_w$ and $\nu_w^*$ }

\label{sec:nuvstaustar}

In this section we establish Theorem \ref{thm:wKriv}. Here we prefer to work in the setting of multigraphs, rather than weighted graphs (see the discussion in the introduction about their equivalence). Given this correspondence, part \ref{subthm:wKriv:2} of Theorem \ref{thm:wKriv} follows immediately from Krivelevich's original proof applied to multigraphs. To show part \ref{subthm:wKriv:1}, we require the following two preliminary results. Recall that an \emph{edge cut} in a multigraph $G$ is a set $F\subseteq E(G)$ for which there exists a vertex set $W\subseteq V(G)$ where $F$ is the set of edges with one end in $W$ and one end in $V(G)\setminus W$.

\begin{theorem}[Edwards \cite{Edw}]
\label{thm:edge cut}
If $G$ is a multigraph with $e$ edges, then $G$ has an edge cut of size at least $e/2 + \sqrt{e/8} - 1 $.
\end{theorem}

\begin{lemma}
\label{lem:bigindep}
If\/ $G$ is a triangle-free multigraph with $v$ vertices, then $G$ has an independent vertex set of size at least $\sqrt{v/3}$.
\end{lemma}

\noindent{\it Proof sketch.} To see this, observe that either $G$ has a vertex of degree at least $\sqrt{v/3}$ whose neighbours form an independent set, or the greedy algorithm (choosing a vertex to add to the independent set and deleting its neighbours) yields an independent set of the desired size.
\hfill$\Box$

\bigskip
The factor of $\sqrt{3}$ in the above lemma can easily be improved, but is all we require for our theorem.  In fact, a difficult theorem due to Kim \cite{Kim} gives a best possible improvement to the above lemma, showing that the conclusion may be improved to find an independent set of size $\Omega( \sqrt{v \log v} )$.  However, applying his theorem instead of our easy lemma would not improve the bound we achieve.

The only additional ingredient required for the proof is the notion of complementary slackness.  This is a fundamental property in the world of linear programming, and can be found in most books on the subject, such as \cite{Sch}.
With this we are ready to establish property \ref{subthm:wKriv:1} in Theorem~\ref{thm:wKriv}.

\begin{theorem} If\/ $G$ is a multigraph, then $\tau(G) \leq 2\nu^*(G)-\sqrt{\nu^*(G)/6}+1$.
\end{theorem}

\begin{proof}
Let $E=E(G)$ and let $\cT=\cT(G)$. Fix an optimal fractional transversal $g\colon E \to \mathbb{R}$ and an optimal fractional packing $f\colon \cT \to \mathbb{R}$ so that we have $f(\cT)=\nu^*(G)=\tau^*(G)= g(E).$  If a triangle $t\in \cT$ has edges $e_1, e_2, e_3$, we say that $t$ is \DEF{tight}  if $g(e_1)+g(e_2)+g(e_3)=1$. Similarly, we say that an edge $e\in E$ is  \DEF{tight} if $\sum_{t\in \cT, e\in t}f(t)=1$. Observe that by the complementary slackness for dual linear programs \eqref{eq:s1} and \eqref{eq:s2}, $g(e)>0$ implies that $e$ is tight, and $f(t)>0$ implies that $t$ is tight. Hence, when computing $f(\cT)$ and $g(E)$, we need only consider tight triangles and tight edges, respectively.

Let $Z$ denote the set of all edges $e\in E$ having $g(e)=0$. All edges in $E\setminus Z$ are tight by complementary slackness, and we partition them into four sets $A, B, C, D$ as follows: for every edge $e\in E\setminus Z$, let $e\in A$ if $0< g(e)<1/2$; $e\in B$ if $g(e)=1/2$; $e\in C$ if $1/2<g(e)<1$; $e\in D$ if $g(e)=1$. The tight triangles of $G$ can then be partitioned into five sets, $\cT_1, \ldots, \cT_5$ where for $i\in\{1, 2, 3\}$,  a tight triangle $t$  is a member of $\cT_i$ if $t$ has exactly $i$ edges in $A$. Since $A$ is a set of tight edges, if we let $|A|=a$ this immediately implies that
\begin{equation}\label{eq:a}
a= \sum_{e\in A} \Bigl(\,\sum_{t\in \cT, e\in t} f(t) \Bigr)  = f(\cT_1)+2f(\cT_2)+3f(\cT_3).
\end{equation}
Tight triangles with no edges in $A$ have either two edges in $Z$ and one edge in $D$, or one edge in $Z$ and two edges in $B$. Let $\cT_4$ denote the former set and let $\cT_5$ denote the latter. Note that each triangle in $\cT_1$ has one edge in each of $A$, $Z$ and $B\cup C$, and each triangle in $\cT_2$ has two edges in $A$ and one edge in $B\cup C$. The triangles of $\cT_3$ and $\cT_4$ are the only tight triangles with no edges in $B\cup C$. Since $B\cup C \cup D$ is a set of tight edges, if we let $|B|=b$, $|C|=c$, and $|D|=d$, we thus get
\begin{equation}\label{eq:bplusc}
b+c=\sum_{e\in B\cup C} \Bigl(\,\sum_{t\in \cT, e\in t} f(t)  \Bigr) = f(\cT_1)+f(\cT_2)+2f(\cT_5).
\end{equation}
and
\begin{equation}\label{eq:d}
d=f(\cT_4).
\end{equation}
We now use \eqref{eq:a}, \eqref{eq:bplusc} and \eqref{eq:d} to get a lower bound for $\nu^*(G)$, as follows:
\begin{align*}
\nu^*(G)&=f(\cT)\\
  &= f(\cT_1)+f(\cT_2)+f(\cT_3)+f(\cT_4)+f(\cT_5)\\
  &\geq \left( \tfrac{1}{4}f(\cT_1)+\tfrac{1}{2}f(\cT_2)+ \tfrac{3}{4}f(\cT_3)\right)+\left( \tfrac{1}{2}f(\cT_1)+ \tfrac{1}{2}f(\cT_2)+f(\cT_5) \right)+f(\cT_4)\\
  &= \frac{a}{4} +\frac{b+c}{2}+d.
\end{align*}
To complete the proof we will show that $G$ has a transversal of size at most
\begin{equation}\label{eq:transize}
2\left( \frac{a}{4} +\frac{b+c}{2}+d \right)-\frac{1}{\sqrt{6}}\sqrt{\frac{a}{4} +\frac{b+c}{2}+d}\,+1
\end{equation}
and use the fact that the function $2x - \sqrt{\tfrac{x}{6}}+1$ is increasing for $x\ge \tfrac{1}{4}$ combined with the inequality of the previous paragraph.

Let $H$ be the graph with vertex set $B$, where two elements $e, e'$ of $B$ are adjacent if $e,e'$ are two edges of some tight triangle in $G$. Note that a tight triangle with two edges in $B$ must have its third edge in $Z$. Since $g$ is a fractional transversal, no triangle can have all three edges in $Z$, which implies that $H$ is triangle-free. Hence, by Lemma \ref{lem:bigindep}, $H$ has an independent vertex set $I\subseteq B$ of size at least $\sqrt{b/3}$. We claim that $(B\setminus I)\cup C\cup D$, along with the complement of any edge-cut in $G'=G[A\cup I]$, is a transversal of $G$. To see this, first note that a triangle in $G$ has at most 2 edges in $Z$, and if it has exactly 2 edges in $Z$ then its third edge is in $D$. If a triangle in $G$ contains an edge of $Z$ but no edge of $D$ then it contains either an edge of $C$, or two edges of $B$ (in which case it is tight and thus contains at least one edge of $B\setminus I$). Any triangle containing no edges of $Z$ or $D$ either contains an edge of $(B\setminus I)\cup C$ or only edges of $A\cup I$ (in which case at most two of them are in an edge-cut of $G'$).

By Theorem \ref{thm:edge cut}, if $G'$ has $e'=|A|+|I|$ edges, then it has an edge-cut $S$ of size at least $e'/2+\sqrt{e'/8}-1$. Let $R$ be the edge-complement of $S$ in $G'$, and let $L=(B\setminus I)\cup C \cup D \cup R$. Then $L$ is a transversal of $G$, and moreover,
\begin{align*}
|L| &\leq  (b-| I |) +c+d+ \left(e'/2 - \sqrt{e'/8} + 1 \right)\\
  &= (b-|I|) +c+d  + \left(\frac{a+|I|}{2} - \sqrt{\frac{a+| I |}{8}} \right) + 1\\
  &\leq \frac{a}{2}+b+ c+d -\left(\frac{|I|}{2}+ \sqrt{\frac{a}{8}} \, \right) + 1\\
  &\leq \frac{a}{2}+b+ c +2d - \frac{1}{\sqrt{6}}\left(\sqrt{\frac{b}{2}}+ \sqrt{\frac{3a}{4}}+d \right) + 1\\
  &\leq  2\left(\frac{a}{4}+\frac{b+ c}{2}+d\right)  - \frac{1}{\sqrt{6}}\sqrt{\frac{3a}{4}+\frac{b}{2} + d} + 1\ .
\end{align*}
If $a \geq c$, then the right-hand side in the above inequality is at most \eqref{eq:transize}, as desired. Thus it suffices to prove that $a\geq c$ is implied by the optimality of the fractional transversal $g$. To see this, note that if $a<c$, then we may define $g_{\varepsilon}\colon E\to \mathbb{R}$ from $g$ by  adding $\varepsilon>0$ to each edge in $A$ and subtracting $\varepsilon$ from each edge in $C$. Every tight triangle has at least as many edges in $A$ as in $C$, so $g'$ is a transversal of the tight triangles for every $\varepsilon>0$. Given that the remaining triangles are not tight, there is a sufficiently small value of $\varepsilon$ such that $g'$ is a fractional transversal of $G$. However then $g'(E)<g(E)$, so the optimality of $g$ yields the desired contradiction.
\end{proof}

This result is best possible up to a logarithmic factor. To see this, let $G$ be a graph formed by taking any $n$-vertex triangle-free graph $H$ and adding one apex vertex completely joined to $H$. Taking each edge incident to the apex with value $1/2$ gives a fractional transversal in $G$. This shows that $\tau^*(G) \leq n/2$.
Suppose now that $R$ is a transversal for $G$. If $R$ contains an edge $xy$ of $H$, we may replace this edge by the edge joining the apex with $x$, and this would still be a transversal. Therefore, we may assume that $R$ contains only edges incident with the apex. Let $U\subseteq V(H)$ be the set of endvertices of the edges in $R$ (excluding the apex). Since $R$ is a transversal, the set $V(H)\setminus U$ is an independent set in $H$. If $H$ is a \DEF{Ramsey graph} (a largest triangle-free graph without an independent set of size $k$), then the bound for triangular Ramsey numbers $r(3,k)$ (the aforementioned result of Kim \cite{Kim}) shows that $k=\Theta(\sqrt{n\log n})$ where $n = r(3,k)-1$. In particular, $|V(H)\setminus U| < k$, so $|R| = |U| \ge n - \Theta(\sqrt{n\log n}\,)$. This implies that $\tau(G) \ge 2\,\tau^*(G) - \Theta\big(\sqrt{\tau^*(G)\log \tau^*(G)}\,\big)$.

\section{Graphs on a surface}\label{sec:surfaces}

In this section we prove that the weighted version of Tuza's conjecture holds for planar graphs by proving a more general statement, Theorem \ref{thm:surface}. Tuza himself proved the unweighted version of the planar case (Theorem \ref{thm:tuza planar}). Our argument is quite similar to his proof, but in some ways the introduction of weights simplifies the situation.

\begin{proof}[Proof of Theorem~\ref{thm:surface}]
Let $G$ and $w$ be a counterexample of Theorem \ref{thm:surface} so that $|E| + w(E)$ is minimum.  We shall establish properties of $G,w$ in a few steps. Let us observe that none of these properties uses embeddability in a surface, but all reductions used in the proofs preserve embeddability and do not introduce new triangles.

\medskip
\begin{enumerate}[label=\textup(\arabic*), topsep=0pt, leftmargin=*]
\item\label{enu:w(e)>0}
  $w(e) > 0$ for every $e \in E$.
\end{enumerate}

Suppose (for a contradiction) that $w(e)=0$, and consider the graph $G' = G - e$ and the weight function $w'$ obtained by restricting $w$ to $E \setminus \{e\}$.  Since adding $e$ to a transversal of $G'$ yields a transversal of $G$ with the same weight, we have $2 \nu_w(G) = 2 \nu_{w'}(G') \ge \tau_{w'}(G') = \tau_w(G)$ which is a contradiction.

\medskip
\begin{enumerate}[resume*]
\item\label{enu:e in two triangles}
  Every $e \in E$ is in at least two triangles.
\end{enumerate}

If $e$ is not in any triangle, then $G - e$ is a smaller counterexample, which is contradictory.  Next suppose that $e$ is in exactly one triangle with edge set $\{e,f_1,f_2\}$.  Now modify $w$ to form a new weight function $w'$ by setting $w'(e) = w(e) -1$ and $w'(f_i) = w(f_i) - 1$ for $i=1,2$.  Let $R$ be an (inclusion-wise) minimal transversal of $G$ with $w'(R) = \tau_{w'}(G)$.  Clearly, if $R$ contains $f_1$ or $f_2$, then it does not contain $e$. Thus, we conclude:
\[2 \nu_w(G) \ge 2 \nu_{w'}(G) + 2 \ge \tau_{w'}(G) + 2 = w'(R) + 2 \ge w(R) \ge \tau_w(G). \]
This contradiction proves \ref{enu:e in two triangles}.

\medskip
\begin{enumerate}[resume*]
\item\label{enu:w(e)<=1}
  If $e\in E$ is in exactly two triangles, then $w(e) \le 1$.
\end{enumerate}

Suppose (for a contradiction) that \ref{enu:w(e)<=1} fails and the edge $e$ with $w(e) \ge 2$ is in exactly two triangles with edge sets $\{e,f_1,f_2\}$ and $\{e,f_3,f_4\}$.  Next, modify $w$ to form a new weight function $w'$ by setting
$w'(e) = w(e) - 2$ and $w'(f_i) = w(f_i) - 1$ for $1 \le i \le 4$.  Let $R$ be a minimal transversal of $G$ with $w'(R) = \tau_{w'}(G)$.  By minimality, $R$ cannot contain $e$ and at least one of $f_1,f_2$ and at least one of $f_3,f_4$.
It follows from this that $w(R) \le w'(R) + 4$.  This gives us
\[2 \nu_w(G) \ge 2 \nu_{w'}(G) + 4 \ge \tau_{w'}(G) + 4 = w'(R) + 4 \ge w(R) \ge \tau_w(G) \]
which is a contradiction.

\medskip
\begin{enumerate}[resume*]
\item\label{enu:N(v) no induced cycle}
  $G$ does not contain a vertex $v$ so that the set $N(v)$ of its neighbors induces a cycle.
\end{enumerate}

Suppose (for a contradiction) that \ref{enu:N(v) no induced cycle} is false and that $N(v)$ induces a cycle with (cyclic) order $u_1, u_2, \ldots, u_k$.  Note that by \ref{enu:w(e)>0} and \ref{enu:w(e)<=1} we have $w(vu_i) = 1$ for
every $1 \le i \le k$.  Now set $G' = G - v$ and let $w'$ be the function obtained from the restriction of $w$ to $E(G')$ by setting $w'(u_{2i-1} u_{2i}) = w(u_{2i-1} u_{2i}) - 1$ for $1 \le i \le \lfloor \frac{k}{2} \rfloor$.  Let $R'$ be a transversal of $G'$ with $w'(R') = \tau_{w'}(G')$.  If $R'$ does not use any of the edges
$u_{2i-1} u_{2i}$ for $1 \le i \le \lfloor \frac{k}{2} \rfloor$ then we may extend $R'$ to a transversal $R$ of $G$ by adding $\lceil \frac{k}{2} \rceil$ edges of the form $v u_j$ and we have
that $w(R) \le w'(R') + \lceil \frac{k}{2} \rceil \le w'(R') + 2 \lfloor \tfrac{k}{2} \rfloor$.  On the other hand, if $R'$ uses an edge of the form $u_{2i-1} u_{2i}$, then we may extend $R'$ to a transversal $R$
of $G$ by adding  $\lfloor \frac{k}{2} \rfloor$ edges of the form $v u_i$.  In this case we have $w(R) \le w(R') + \lfloor \frac{k}{2} \rfloor \le w'(R') + 2 \lfloor \frac{k}{2} \rfloor$.  This gives us
\begin{align*}
2 \nu_w(G) & \ge 2 \nu_{w'}(G') + 2 \lfloor \tfrac{k}{2} \rfloor \ge \tau_{w'}(G') + 2 \lfloor \tfrac{k}{2} \rfloor\\
  & = w'(R') + 2\lfloor \tfrac{k}{2} \rfloor \ge w(R) \ge \tau_w(G)
\end{align*}
which is a contradiction.

\bigskip

Let us now consider the embedding of $G$. If every triangle is facial (i.e., it bounds a face),
then it follows from \ref{enu:e in two triangles} that every edge of $G$ is in exactly two facial triangles and then applying \ref{enu:N(v) no induced cycle} to any vertex gives us a contradiction.  Otherwise, we may choose a non-facial surface-separating triangle $t$ so that the number of edges in one of the surface components of $t$, say $S$, is minimal.  Now every edge properly inside $S$ is in exactly two facial triangles (again by \ref{enu:e in two triangles}) and then applying \ref{enu:N(v) no induced cycle} to any vertex properly inside $S$ gives us a contradiction.  This completes the proof.
\end{proof}

\section{Comparing $\tau_w$ and $\nu_w$}

\label{sec:tau vs nu}

In this section we will establish Theorem~\ref{thm:whaxell}. As in Section \ref{sec:nuvstaustar}, we prefer to work in the setting of multigraphs, rather than weighted graphs (see the discussion in the introduction about their equivalence).

\begin{theorem}\label{thm:tauvsnu}
  Let\/ $G$ be a multigraph. Then $~\tau(G)\leq \left(3-\tfrac{2}{25}\right)\nu(G)$.
\end{theorem}

In the case of (unweighted) simple graphs, Haxell~\cite{Hax} proved that $\tau \leq c \nu$, where $c \approx 2.866$. Our constant $3-\tfrac{2}{25}=2.92$ is slightly larger. The rest of this section is devoted to the proof of Theorem~\ref{thm:tauvsnu}. The proof structure follows roughly the same lines as in~\cite{Hax}, and we have kept similar notation when it was possible. The proofs of Lemmas~\ref{lem:3/2+5/2gamma+2beta} and~\ref{lem:3+3delta-beta} are taken directly from \cite{Hax}, and are included here for completeness.

\bigskip

We say that a family $\mathcal{F}$ of triangles in a graph is \DEF{independent} if the elements of $\mathcal{F}$ are pairwise edge-disjoint.
Let $\cB$ be an independent family of triangles in $G$ of size $\nu=\nu(G)$. We say that a triangle in $G$ is of type $(\cB,i)$ if it has exactly $i$ edges in common with the set $E[\cB]$. Note that every triangle in $G$ is of type $(\cB,i)$ for some $i\in\{1,2,3\}$. Let $\cB_1$ be an independent family of triangles of type $(\cB,1)$ of maximum size in $G$, and let $\gamma$ be defined by $|\cB_1|=\gamma\nu$.

\begin{lemma}\label{lem:3-2/3gamma}
  $\tau(G)\leq(3-\frac{2}{3}\gamma)\nu$.
\end{lemma}
\begin{proof}
For each $T\in\cB_1$, let $\hat{T}$ denote the triangle in $\cB$ that shares an edge with $T$, let $e(T)$ denote the edge shared by $T$ and $\hat{T}$, let $v(T)$ be the unique vertex of $T$ which is not incident to $e(T)$, let $\hat{v}(T)$ be the unique vertex of $\hat{T}$ which is not incident to $e(T)$, and let $E'(T)$ be the set of edges between $v(T)$ and $\hat{v}(T)$ (see Figure~\ref{fig:B1}).
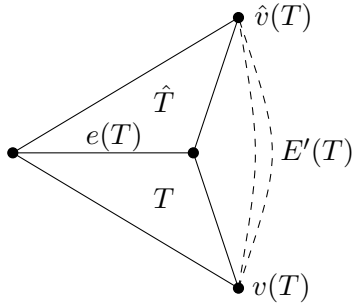
\begin{figure}[htb]
  \begin{center}
\begin{tikzpicture}[scale=0.3]
  \path (0,0) coordinate (x);
  \path (8,0) coordinate (y);
  \path (10,6) coordinate (v');
  \path (10,-6) coordinate (v);
  \path (11,0) coordinate (m1);
  \path (12,0) coordinate (m2);

  \draw (x) -- (y) -- (v) -- (x) -- (v') -- (y);
  \draw[dashed] (v) .. controls (m1) .. (v');
  \draw[dashed] (v) .. controls (m2) .. (v');

  \drawvertex{(x)};
  \drawvertex{(y)};
  \drawvertex{(v)};
  \drawvertex{(v')};

  \drawtext {(6.7,2.4)} {$\hat{T}$};
  \drawtext {(6.7,-2)} {$T$};
  \drawtext {(4.5,0.7)} {$e(T)$};
  \drawtext {(v)++(1.9,0)} {$v(T)$};
  \drawtext {(v')++(2.0,0)} {$\hat{v}(T)$};
  \drawtext {(m2)++(1.5,0)} {$E'(T)$};
\end{tikzpicture}
\end{center}
  \caption{Triangles $T$ and $\hat{T}$.}
  \label{fig:B1}
\end{figure}
Let $\hat{\cB}_1=\set{\hat{T}}{T\in\cB_1}$. From the maximality of $\cB$ it follows that $|\hat{\cB}_1|=|\cB_1|$. Since every family of triangles of the form $(\cB\sm\hat{\cB}_1)\cup\set{T\text{ or }\hat{T}}{T\in\cB_1}$ is an independent family of triangles of size $\nu(G)$, for every triangle $S$ that is edge-disjoint from $\cB\sm\hat{\cB}_1$, there exists $T\in\cB_1$ such that $S$ shares an edge with $T$ as well as with its counterpart $\hat{T}\in\hat{\cB}_1$. Then such a triangle contains either $e(T)$ or an edge from $E'(T)$. Consequently, the set $C=C_1\cup C_2$, where
\begin{align*}
  C_1 & = E[\cB\sm\hat{\cB}_1] \cup \set{e(T)}{T\in\cB_1}\quad\mbox{and}\\
  C_2 & = \textstyle\bigcup\limits_{T\in\cB_1}E'(T),
\end{align*}
is a transversal of $G$. We will show that $|C|\leq(3-\frac{2}{3}\gamma)\nu$. Clearly, $|C|\leq|C_1|+|C_2\sm C_1|=(3-2\gamma)\nu+|C_2\sm C_1|$.
We define the set $\cJ=\set{T\in\cB_1}{E'(T)\subseteq C_1}$, and we define $\gamma_0$ by $|\cJ|=\gamma_0\nu$.

Now consider a triangle $U\in\cB_1$. If $|E'(U)\sm C_1|\neq0$ then $U\in\cB_1\sm\cJ$. Consider the case that $|E'(U)\sm C_1|\geq2$. So there are edges $f_1,f_2\in E'(U)\sm C_1$. Then these two edges together with the four edges in $(E(U)\cup E(\hat{U})) \setminus \{e(U)\}$ form two edge-disjoint triangles, say $T_1(U)$ and $T_2(U)$. Hence at least one of $f_1,f_2$, say $f_1$, must belong to $E[\cB]$; otherwise $(\cB\sm\{\hat{U}\})\cup\{T_1(U),T_2(U)\}$ would be an edge-disjoint family of triangles, contradicting the maximality of $\cB$. Since $f_1\notin C_1$, we conclude that $f_1\in E(\hat{T})\sm e(T)$ for some triangle $T\in\cB_1$. If $f_2\notin E(T)\sm\{e(T)\}$, then $(\cB\sm\{\hat{U},\hat{T}\})\cup\{T,T_1(U),T_2(U)\}$ is a family of edge-disjoint triangles, contradicting the maximality of $\cB$. So $f_2\in E(T)\sm\{e(T)\}$. Since $f_1$ and $f_2$ have the same endpoints, it follows that the edge $f_2$ is uniquely determined and that $v(T)=\hat{v}(T)$. This implies that $E'(T)=\emptyset$ and, therefore, $T\in\cJ$. Since we could not get the same conclusion for a third edge in $E'(U)\setminus C_1$, this also implies that  $|E'(U)\sm C_1|=2$.

The above proof shows that, for any triangle $U\in\cB_1$, the set $E'(U)\sm C_1$ contains at most two edges and, therefore, $|C_2\sm C_1|\leq2(\gamma-\gamma_0)\nu$. Moreover, if $E'(U)\sm C_1$ contains two edges, then one of them belongs to $E[\cJ]\sm C_1$. Therefore, we also have
$|C_2\sm C_1|\leq(\gamma-\gamma_0+2\gamma_0)\nu=(\gamma+\gamma_0)\nu$. Consequently, we have
$$
  |C_2\sm C_1|\leq\frac{1}{3}(2\gamma-2\gamma_0)\nu+\frac{2}{3}(\gamma+\gamma_0)\nu=\frac{4}{3}\gamma\nu
$$
and hence $|C|\leq (3-2\gamma)\nu + |C_2\setminus C_1| \leq (3-\frac{2}{3}\gamma)\nu$.
\end{proof}

\medskip

Next, we let $G'=G-E[\cB_1]$. Note that $\nu(G')=(1-\gamma)\nu$, and that all triangles in $G'$ are of type $(\cB,2)$ or of type $(\cB,3)$. Let $\cB_2$ be an independent family of triangles of type $(\cB,2)$ of maximum size in $G'$, and let $\beta$ be such that $|\cB_2|=\beta\nu$.

\begin{lemma}\label{lem:3/2+5/2gamma+2beta}
  $\tau(G)\leq(\frac{3}{2}+\frac{5}{2}\gamma+2\beta)\nu$.
\end{lemma}

\begin{proof}
  Let $H=G[E[\cB]\sm(E[\cB_1]\cup E[\cB_2])]$. Maximality of $\cB_1$ and $\cB_2$ imply that every type $(\cB,1)$ triangle in $G$ contains an edge of $E[\cB_1]$, and every triangle of type $(\cB,2)$ contains an edge of $E[\cB_1]\cup E[\cB_2]$. Consequently, every triangle that is edge-disjoint from $E[\cB_1]\cup E[\cB_2]$, is
  of type $(\cB,3)$ and, therefore, also a triangle in $H$. Hence, if $C'$ is a transversal of $H$, then the set
  $$C = E[\cB_1] \cup E[\cB_2] \cup C'$$
  is a transversal of $G$. Now, we know that $H$ has a bipartite subgraph $S$ with at least $\frac{1}{2}|E(H)|$ edges. Thus, $C' = E(H)\setminus E(S)$ is a transversal of $H$. It follows that $|C'|\leq\frac{1}{2}|E(H)|=\frac{1}{2}(3-\gamma-2\beta)\nu$. Thus, $|C|\leq(\frac{3}{2}+\frac{5}{2}\gamma+2\beta)\nu$ as desired.
\end{proof}

We let $\cB'$ be an independent family of triangles in $G'$ of maximum size, subject to the condition that $|E[\cB']\sm E[\cB]|\geq\beta\nu$. We know that such a family exists, because $\cB_2$ satisfies the condition. Observe that $\cB'$ is an inclusion-maximal independent family of triangles in~$G'$. We define $\alpha$ by $|\cB'|=\alpha\nu$.

From now on, we use the same notation as in Figure~\ref{fig:B1}, but where the set $\cB'$ plays the role that the set $\cB$ was playing before. More precisely, if $T$ is a triangle of type $(\cB',1)$ in $G'$, we let $\hat{T}$ be the triangle in $\cB'$ that shares an edge with $T$, we let $e(T)$ be the edge shared by $T$ and $\hat{T}$,
we let $v(T)$ be the unique vertex of $T$ which is not incident to $e(T)$, we let $\hat{v}(T)$ be the unique vertex of $\hat{T}$ which is not incident to $e(T)$,
and we let $E'(T)$ be the set of edges between $v(T)$ and $\hat{v}(T)$ in $G'$.

We let $\cB'_1$ be an independent family of triangles of type $(\cB',1)$  in $G'$ such that for each $T\in\cB_1'$, we have $e(T)\not\in E[\cB]$, and such that $\cB_1'$ has maximum cardinality with these properties. We define $\delta$ by $|\cB_1'|=\delta \nu$.

\begin{lemma}\label{lem:switch}
Let $\mathcal{S}$ be any subset of $\cB_1'$. Then the family of triangles:
\begin{equation}
\label{eq:switch}
\tilde \cB' = \mathcal{S} \cup \cB' \setminus \set{\hat{T}}{T\in \mathcal{S}} ,
\end{equation}
is an independent family of triangles in $G'$ such that $| E[\tilde \cB']\setminus E[\cB] | \geq \beta \nu$. Moreover, $\tilde\cB'$ has maximum size with this property.
\end{lemma}

\begin{proof}
Let $T\in\mathcal{S}$. By definition of $\cB_1'$, we have $e(T)\not\in E[\cB]$, and since there are no triangles in $G'$ of type $(\cB,1)$, we have $E(T)\cup E(\hat{T})\setminus \{e(T)\} \subset E[\cB]$. This implies that $E[\tilde \cB']\setminus E[\cB] = E[\cB']\setminus E[\cB]$, which proves the first assertion of the lemma. The second assertion is immediate since $|\tilde \cB'|=|\cB'|$.
\end{proof}

\begin{lemma}\label{lem:3+3delta-beta}
$\tau\leq (3\gamma+3\delta+3\alpha-\beta)\nu$.
\end{lemma}
\begin{proof}
We let
\[
C= E[\cB_1]\cup E[\cB_1'] \cup \big(E[\cB]\cap E[\cB']\big).
\]
Then clearly $|C|\leq (3\gamma + 3\delta + 3\alpha- \beta)\nu$. To complete the proof, it suffices to show
that $C$ is a transversal of $G$. Since $E[\cB_1] \subseteq C$, it suffices to prove that every triangle in $G'$ has an edge in $C$.

First, let $T$ be a triangle of type $(\cB',1)$ in $G'$. If $e(T)\in E[\cB]$, then $e(T) \in E[\cB]\cap E[\cB']$, so $E[T]$ intersects $C$. Else, we know that $E[T]$ intersects $E[\cB_1']$ since otherwise adding $T$ to $\cB_1'$ would contradict the maximality of $\cB_1'$; so every triangle of type $(\cB',1)$ in $G'$ intersects $C$.

Now, let $T$ be a triangle of type $(\cB',2)$ or $(\cB',3)$ in $G'$. Since in $G'$ all triangles have type $(\cB,2)$ or $(\cB,3)$, $E[T]$ necessarily contains at least one edge in $E[\cB]\cap E[\cB']$. Therefore $E[T]$ intersects $C$, which concludes the proof that $C$ is a transversal of~$G$.
\end{proof}

In $G'$, we define the set of triangles $\hat{\cB}_1'=\set{\hat{T}}{T \in \cB_1'}\subseteq\cB'$ and we consider the edge-set $E_0$ defined by
\begin{equation}
  E_0 = E[\cB'\setminus\hat{\cB}_1'] \cup \textstyle\bigcup\limits_{T \in \cB_1'} \{e(T)\}.
  \label{eq:s0}
\end{equation}
We are now going to define a subset $\cI$ of $\cB_1'$, and a function $f\colon \cI \to 2^{E[G']}$ that associates to each triangle in $\cI$ a set of edges of $G'$. The set $\cI$ and the function $f$ are chosen simultaneously according to the following properties:

-- for each $T\in \cI$, we have $f(T) \subseteq E'(T)\setminus E[\cB']$ and $|f(T)|=2$;

-- the sets $(f(T))_{T\in\cI}$ are pairwise disjoint, and so are the sets $(E(T))_{T\in\cI}$;

-- for any $T,U\in\cI$, the sets $f(T)$ and $E(U)$ are disjoint;

-- $\cI$ has maximum cardinality subject to these properties.

\smallskip

By Lemma~\ref{lem:switch}, any set $\tilde\cB'$ of the form of Eq.~\eqref{eq:switch} satisfies the hypotheses of the definition of $\cB'$. Replacing $\cB'$ by $\tilde \cB'$ may change the cardinality of the set $\cI$ defined above. From now on we will assume that, among all sets $\tilde\cB'$ of the form of Eq.~\eqref{eq:switch}, the set $\cB'$ is the one for which the set $\cI$ has the maximum cardinality. We let $\hat{\cI}=\set{\hat{T}}{T\in\cI}$ and we define $\eta$ by $|\cI|=|\hat{\cI}|=\eta\nu$.

\begin{lemma}\label{lem:alpha+eta}
$\alpha+\eta \leq 1-\gamma$.
\end{lemma}
\begin{proof}

For $T\in \cI$, the two edges of $f(T)$, together with the four edges in $(E(T)\cup E(\hat{T})) \setminus \{e(T)\}$ form two edge-disjoint triangles, say $T_1(T)$ and $T_2(T)$. Now, by definition of $\cI$, the family of triangles
$$
\cA= \set{T_1(T),T_2(T)}{T\in\cI} \cup \big( \cB'\setminus\hat{\cI}  \big)
$$
is an edge-disjoint family of triangles in $G'$.
Since $|\cA|=|\cB'|+|\hat{\cI}|=(\alpha+\eta)\nu$, we have
$
(\alpha+\eta)\nu \leq \nu(G') = (1-\gamma) \nu.
$
\end{proof}

Now, we let
$\cK=\set{T\in\cB_1'}{E'(T)\subseteq E_0}$, we let
$\hat{\cK}=\set{\hat{T}\in\cB'}{T\in\cK}$, and we define $\delta_0$ by $|\cK|=|\hat{\cK}|=\delta_0\nu$.

\begin{lemma}\label{lem:lastbut}
$\tau \leq (3\gamma+3\alpha-2\delta_0)\nu \leq (3-3\eta-2\delta_0)\nu$.
\end{lemma}

\begin{proof}
We define the set of edges
$$
C= E[\cB_1] \cup E[\cB'\setminus \hat{\cK}] \cup \set{e(T)}{T\in\cK},
$$
and we observe that $|C|\leq (3\gamma+3\alpha-2\delta_0)\nu\leq (3-3\eta-2\delta_0)\nu$ (for the second inequality we have used Lemma~\ref{lem:alpha+eta}).

Now we prove that $C$ is a transversal of $G$. Since $E[\cB_1]\subseteq C$, it suffices to prove that $C$ is a transversal in $G'$.
By Lemma~\ref{lem:switch}, every family of triangles of the form $(\cB'\sm \hat{\cK})\cup\set{T\text{ or }\hat{T}}{T\in K}$ is an inclusion-maximal independent family of triangles in $G'$. Therefore for every triangle $U$ in $G'$ that is edge-disjoint from $\cB'\sm \hat{\cK}$, there exists $T\in\cK$ such that $U$ shares an edge with $T$ as well as with its counterpart $\hat{T}\in\hat{\cK}$. Then the triangle $U$ contains either $e(T)$ or an edge from $E'(T)$. In the first case, $E(U)$ intersects $C$; in the second case, since by the definition of~$\cK$, $E'(T)$ is contained in $E_0\subseteq C$,  $E(U)$ intersects $C$ as well.
\end{proof}

\begin{lemma}\label{lem:last}
$\tau \leq (3-\delta+4\eta+\delta_0)\nu$.
\end{lemma}

\begin{proof}
First we define a subset $\cI'$ of $\cB_1'\setminus \cI$ by
$$\cI'=\set{T'\in\cB_1'\sm(\cI\cup\cK)}{E(T')\cap(\textstyle\bigcup\limits_{T\in\cI}f(T)) \neq \emptyset}.$$
We define $\eta'$ by $|\cI'|=\eta'\nu$. Since $\cI'$ is a family of edge-disjoint triangles, it follows that $|\cI'| \le |\bigcup_{T\in\cI}f(T)| = 2|\cI|$, hence $\eta'\leq2\eta$.
Let $\cA=\cB_1'\sm(\cI\cup\cI'\cup\cK)$ and $\hat{\cA}=\set{\hat{T}}{T\in\cA} \subseteq \cB'$.
We now recall the definition \eqref{eq:s0} of the edge-set $E_0$ and define the following set of edges of $G$:
$$
C=E[\cB_1]\cup E_0 \cup E_1 \cup E_2 \cup E_3 \cup E_4,
$$
where
\begin{align*}
  E_1 & = \textstyle\bigcup\limits_{T\in\cI} \big( E(T)\cup E(\hat{T})\cup f(T) \big),\\
  E_2 & = \textstyle\bigcup\limits_{T\in\cI'} E(\hat{T}),\\
  E_3 & = \textstyle\bigcup\limits_{T\in\cK} E(\hat{T}),\\
  E_4 & = \textstyle\bigcup\limits_{T\in\cA} E'(T).
\end{align*}

We claim that $C$ is a transversal of $G$.
Since $E[\cB_1]\subseteq C$ and $E[\cB'\sm\hat{\cA}] \subseteq E_0 \cup E_1 \cup E_2 \cup E_3 \subseteq C$, we only have to consider triangles in $G'$ which are edge-disjoint from $\cB'\sm\hat{\cA}$.
By Lemma~\ref{lem:switch}, every family of triangles of the form $(\cB'\sm\hat{\cA})\cup\set{T\text{ or }\hat{T}}{T\in\cA}$ is an inclusion-maximal independent family of triangles in $G'$. Therefore, for every triangle $U$ in $G'$ that is edge-disjoint from $\cB'\sm\hat{\cA}$, there exists $T\in\cA$ such that $U$ shares an edge with $T$ as well as with its counterpart $\hat{T}\in\hat{\cA}$. Then the triangle $U$ contains either $e(T)$ or an edge from $E'(T)$. In the first case, $E(U)$ intersects $E_0$; in the second case, $E(U)$ intersects $E_4$. Hence $C$ is a transversal of $G$.

It remains to show that $|C|\leq(3-\delta+4\eta+\delta_0)\nu$. Clearly,
\begin{align*}
  |C| & \leq |E[\cB_1]| + |E_0| + |E_1\sm E_0| + |E_2\sm E_0| + |E_3\sm E_0|\\
      & \phantom{\leq~} + |E_4\sm(E_0\cup E_1\cup E_2\cup E_3)|\\
      & \leq \big( 3\gamma +3\alpha -2\delta +6\eta +2\eta' +2\delta_0\big) \nu + |E_4\sm(E_0\cup E_1\cup E_2\cup E_3)|.
\end{align*}

Let us consider a triangle $U\in\cA$. We claim that $|E'(U)\sm(E_0\cup E_1\cup E_2\cup E_3)|\leq1$.
Assume to the contrary that there are two distinct edges $f_1,f_2\in E'(U)\sm(E_0\cup E_1\cup E_2\cup E_3)$. If $f_1\in E[\cB']$ then, since $f_1 \notin E_0\cup E_1\cup E_2\cup E_3$, we have $f_1\in E(\hat{T})$ for a triangle $T\in\cA$.
Since $T\notin\cI'$, the set $E(T)\cup E(\hat{T})$ is disjoint from $\set{f(T')}{T'\in\cI}$. Hence the set $\tilde\cB'=(\cB'\cup\{T\})\sm\{\hat{T}\}$ not only satisfies the hypothesis of the definition of $\cB'$, but the sets $\cI$, $\cI'$, $\cK$, $E_0\ldots E_4$, $C$ satisfy the hypotheses of their definition also with respect to $\tilde\cB'$ instead of $\cB'$. Moreover, since $T\notin\cK$, we have $v(T)\neq\hat{v}(T)$ and, therefore, $f_2\notin E(T)\cup E(\hat{T})$. Hence $f_2$ belongs to $\tilde\cB'$ exactly if it belongs to $\cB'$. So we can work with $\tilde\cB'$ instead of $\cB'$ without affecting the involved edge sets or the status of $f_2$. Since $f_1\notin\tilde\cB'$, we may assume in the first place that $f_1\notin\cB'$ and, by a similar argument, that also $f_2\notin\cB'$. However, then we can define $f(U)=\{f_1,f_2\}$ and the set $I\cup\{U\}$ satisfies the hypotheses of the definition of $\cI$, contradicting the maximality of $|\cI|$. This proves the claim.

Consequently, we have $|E_4\sm(E_0\cup E_1\cup E_2\cup E_3)|\leq|\cA|=(\delta-\eta-\eta'-\delta_0)\nu$. Note that this last equality holds because the sets $\cI$, $\cK$, $\cI'$ and $\cA$ are disjoint by definition. Hence we have $|C|\leq (3\gamma +3\alpha -\delta +5\eta +\eta' +\delta_0) \nu$. Using Lemma~\ref{lem:alpha+eta} and the fact that $\eta'\leq2\eta$, we obtain $|C|\leq(3-\delta+4\eta+\delta_0)\nu$.
\end{proof}

\begin{proof}[Proof of Theorem~\ref{thm:tauvsnu}]
Combining inequalities in Lemmas~\ref{lem:3-2/3gamma}--\ref{lem:last}, we have:
\begin{align*}
\left(\tfrac{1}{5}+\tfrac{4}{75}+\tfrac{8}{75}+\tfrac{8}{25}+\tfrac{8}{25}\right) \tau
  & \leq
    \tfrac{1}{5}\left(3-\tfrac{2}{3}\gamma\right)\nu
    +\tfrac{4}{75} \left(\tfrac{3}{2}+\tfrac{5}{2}\gamma+2\beta\right)\nu\\
  & \phantom{\leq~}
    +\tfrac{8}{75} \left(3\gamma+3\delta+3\alpha-\beta\right) \nu
    +\tfrac{8}{25} \left(3\gamma+3\alpha-2\delta_0\right) \nu\\
  & \phantom{\leq~}
    +\tfrac{8}{25}\left(3-\delta+4\eta+\delta_0\right)\nu,
\end{align*}
which gives $\tau\leq \left(\frac{41}{25}+\frac{32}{25}(\gamma+\alpha+\eta)-\frac{8}{25}\delta_0\right) \nu \leq \frac{73}{25} \nu = \left(3-\frac{2}{25}\right)\nu$.
\end{proof}

\section{Integrality}

\label{sec:integrality}


The goal of this section is to establish Theorem \ref{thm:integrality}, the proof of which relies on the family of graphs $\{G_k\}_{k \in {\mathbb N}}$ defined below.
Each graph $G_k$ has two distinguished vertices called \DEF{terminals} which are joined by an edge called the \DEF{terminal} edge. The graph $G_0$ consists of a single edge which is its terminal edge.  For $k \ge 1$ the graph $G_k$ is constructed as follows (see Figure~\ref{fig:gk} below). Start with the 5-wheel graph $W_5$ consisting of a 5-cycle with vertices $v_1,v_2,\ldots,v_5$ and an additional vertex $u$ joined to $v_1,v_2,\dots,v_5$. To obtain $G_k$, take a copy of $G_{k-1}$ for each edge $xy$ of $W_5$ and identify the terminal edge of this copy with $xy$.
We define $v_1,v_2$ to be the terminal vertices of $G_k$ (so the edge $v_1 v_2$ is the terminal edge).

\begin{figure}[htbp]
  \centering
  \includegraphics[height=4cm]{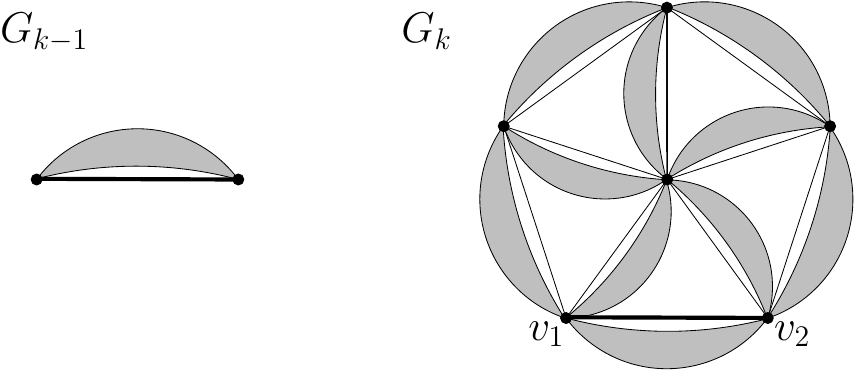}
  \caption{Recursive construction of graphs $G_k$.}
  \label{fig:gk}
\end{figure}

\begin{lemma}
\label{lem:tau*(G_k)}
For every $k \in {\mathbb N}$, we have
\[ \tau^*(G_k) = \nu^*(G_k) = \frac{5}{2^k} \left( \frac{20^k - 1}{19} \right). \]
\end{lemma}

\begin{proof}

For every triangle $t$ in $G_k$ we define the \DEF{height} of $t$ to be the smallest integer $i$ so that $t$ appears in a copy of $G_{i}$ used in the recursive construction.  It is straightforward to verify that for $1\le j \le k$ the graph $G_k$ has exactly $5 \cdot 10^{k-j}$ triangles of height $j$.  We now define the function $f_k$ on the triangles $\mathcal{T}_k$ of $G_k$ by the rule that $f_k(t) = 2^{-j}$ where $j$ is the height of the triangle $t$. We claim that $f_k$ is a fractional packing.
To this end, note that if an edge $e$ appears in a copy of $G_i$ but is not its terminal edge, then it will not appear in any triangles of height greater than $i$. The only way that $e$ can appear in two triangles of height $i$ is if it is the terminal edge of a copy of $G_{i-1}$ that was placed on a spoke of $W_5$ to form $G_i$ (in which case it is certainly not the terminal edge of $G_i$). So,
$$\sum_{t\in \mathcal{T}_k:e\in t}f_k(t)\leq 2^{-1}+2^{-2}+\cdots +2\cdot 2^{-k}=1.$$
Thus, $f_k$ is a fractional packing with value
\[ \sum_{j=1}^k \frac{5}{2^j}  10^{k-j} = \frac{5}{2^k} \sum_{j=1}^k 20^{k-j} = \frac{5}{2^k} \left( \frac{20^k - 1}{19} \right). \]

Next, for every real number $0 \le a \le 1$, we define the function $g_{k,a}$ on the edges of $G_k$.  We let $g_{0,a}$ be the function that assigns the only edge of $G_0$ the value $a$, and for $k \ge 0$ we define $g_{k+1,a}$ recursively as shown in Figure~\ref{fig:trans}.

\begin{figure}[htbp]
  \centering
  \includegraphics[height=5.6cm]{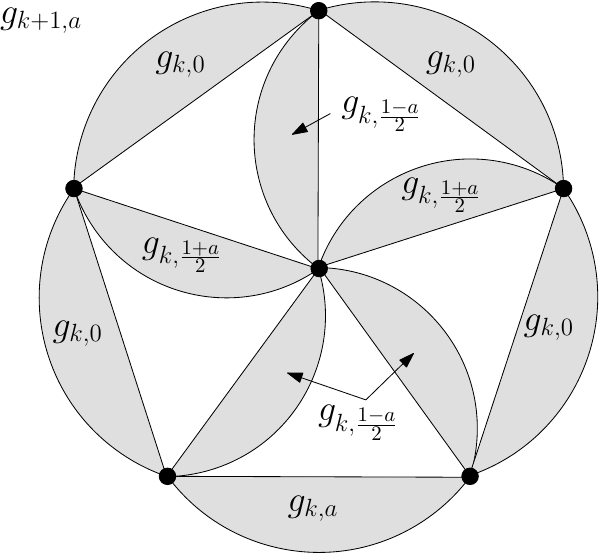}
  \caption{A fractional transversal on $G_{k+1}$.}
  \label{fig:trans}
\end{figure}

Note that the terminal edge of $G_k$ is always assigned $a$ under $g_{k,a}$. From this it is immediate that for every triangle of height $k$, the sum of $g_{k,a}$ over its edges is equal to one.  Then, recursively, the same property holds for every triangle in $G_k$, so $g_{k,a}$ is a fractional transversal.  We claim that the value of $g_{k,a}$ is equal to $\frac{5}{2^k} ( \frac{20^k - 1}{19}) + \frac{a}{2^k}$.  This is immediate for the base case when $k=0$ and then follows inductively from the following computation (here we use $\phi(g_{k,a})$ to denote the value of this fractional transversal):
\begin{align*}
\phi(g_{k,a})
 &= \phi(g_{k-1,a}) + 4 \,\phi(g_{k-1,0}) + 3\, \phi(g_{k-1, \frac{1-a}{2}}) + 2 \,\phi(g_{k-1, \frac{1+a}{2}})  \\
 &= 10 \cdot \frac{5}{2^{k-1}} \left( \frac{20^{k-1}-1}{19}\right) + \frac{ a + 3 \tfrac{1-a}{2} + 2 \tfrac{1+a}{2} }{2^{k-1}} \\
 &= 20 \cdot \frac{5}{2^k} \left( \frac{20^{k-1} - 1}{19} \right) + \frac{5}{2^k} + \frac{a}{2^k} \\
 &= \frac{5}{2^k} \left( \frac{20^k - 1}{19} \right) + \frac{a}{2^k}\,.
 \end{align*}
We now have that $g_{k,0}$ is a fractional transversal of $G_k$ with value $\frac{5}{2^k} ( \frac{20^k - 1}{19} )$ which matches the value of our fractional packing, thus completing the proof.
\end{proof}

\bigskip

\begin{proof}[Proof of Theorem~\ref{thm:integrality}]
It follows from Lemma \ref{lem:tau*(G_k)} that $2^k \tau^*(G_k)$ is an odd integer for every $k \in {\mathbb N}$.  It follows easily from this that
$\tfrac{1}{m}\,\tau_m(G_k) \neq \tau^*(G_k)$ and $\tfrac{1}{m}\,\nu_m(G_k) \neq \nu^*(G_k)$ whenever $m < 2^k$.
\end{proof}

\subsection*{Acknowledgements}
We would like to thank the anonymous referee who found a mistake in the proof of Theorem \ref{thm:wKriv} and offered several suggestions which improved the presentation of the paper.

\bibliographystyle{plain}

\begin{thebibliography}{1}

\bibitem{Ed}
Jack Edmonds.
\newblock Maximum matching and a polyhedron with 0,1-vertices.
\newblock {\em Journal of Research of the National Bureau of Standards}, 69:125--130, 1965.

\bibitem{Edw}
Christopher S. Edwards.
\newblock An improved lower bound for the number of edges in a largest bipartite subgraph.
\newblock {\em Recent advances in graph theory} (Proc. Second Czechoslovak Sympos., Prague, 1974), 167--181, Academia, Prague, 1975.

\bibitem{EGK}
Paul Erd\H{o}s, Andras Gy\'arf\'as, and Yoshiharu Kohayakawa.
\newblock The size of the largest bipartite subgraphs.
\newblock {\em Discrete Math.} 177(1-3):267--271, 1997.

\bibitem{Hax}
Penny Haxell.
\newblock Packing and covering triangles in graphs.
\newblock {\em Discrete Math.}, 195(1-3):251--254, 1999.

\bibitem{Hu}
Te Chiang Hu.
\newblock Multi-commodity network flows.
\newblock {\em Operations Research} 11:344--360, 1963.

\bibitem{Kim}
Jeong~Han Kim.
\newblock The Ramsey number {$R(3,t)$} has order of magnitude {$t^2/\log t$}.
\newblock {\em Random Structures Algorithms}, 7(3):173--207, 1995.

\bibitem{LS}
Laszlo Lov\'asz and \'Akos Seress.
\newblock The cocycle lattice of binary matroids.
\newblock {\em European J. Combin.}, 14(3)3:241--250, 1993.

\bibitem{Kri}
Michael Krivelevich.
\newblock On a conjecture of {T}uza about packing and covering of triangles.
\newblock {\em Discrete Math.}, 142(1-3):281--286, 1995.

\bibitem{MT}
Bojan Mohar and Carsten Thomassen.
\newblock {\em Graphs on Surfaces}.
\newblock Johns Hopkins Studies in the Mathematical Sciences. Johns Hopkins
  University Press, Baltimore, MD, 2001.

\bibitem{Sch}
Alexander Schrijver.
\newblock {\em Theory of Linear and Integer Programming.}
\newblock John Wiley, 1986.

\bibitem{Se}
Paul D. Seymour.
\newblock Sums of circuits.
\newblock {\em Graph theory and related topics} (Proc. Conf., Univ. Waterloo, Waterloo, Ont., 1977) 341--355, Academic Press, New York-London, 1979.

\bibitem{Tuza1}
Zsolt Tuza.
\newblock Conjecture. in: Finite and infinite sets. {V}ol.\ {I}, {II}.
\newblock In A.~Hajnal, L.~Lov{\'a}sz, and V.~T. S{\'o}s, editors, {\em
  Proceedings of the sixth {H}ungarian combinatorial colloquium held in {E}ger,
  {J}uly 6--11, 1981}, volume~37 of {\em Colloquia Mathematica Societatis
  J\'anos Bolyai}, page 888. North-Holland Publishing Co., Amsterdam, 1984.

\bibitem{Tuza2}
Zsolt Tuza.
\newblock A conjecture on triangles of graphs.
\newblock {\em Graphs Combin.}, 6(4):373--380, 1990.

\end{thebibliography}

\end{document}